\documentclass[10pt,reqno]{amsart}
\usepackage[centertags]{amsmath}

\theoremstyle{plain} 

\newtheorem{thm}{Theorem}

\theoremstyle{definition}

\theoremstyle{remark}

\numberwithin{equation}{section}

\DeclareMathSymbol{\R}{\mathalpha}{AMSb}{"52}
\DeclareMathSymbol{\C}{\mathalpha}{AMSb}{"43}

\newcommand{\beq}{\begin{equation}}
\newcommand{\eeq}{\end{equation}}

\newcommand{\set}[1]{\left\{#1\right\}}

\newcommand{\pd}{\,\partial}

\newcommand{\bd}{\begin{description}}
\newcommand{\ed}{\end{description}}
\newcommand{\beqr}{\begin{eqnarray}}
\newcommand{\eeqr}{\end{eqnarray}}
\newcommand{\beqt}{\begin{equation}}
\newcommand{\eeqt}{\end{equation}}

\begin{document}

\title[]{Invariants associated with linear ordinary differential equations}
\author[]{J C Ndogmo}

\address{PO Box 2446
Bellville 7535\\
South Africa\\
}
\email{ndogmoj@yahoo.com}

\begin{abstract}
We apply a novel method for the equivalence group and its
infinitesimal generators to the investigation of invariants of
linear ordinary differential equations. First, a comparative study
of this method is illustrated by an example. Next, the method is
used to obtain the invariants of low order linear ordinary
differential equations, and the structure invariance group for an
arbitrary order of these equations. Other properties of these
equations are also discussed, including the exact number of their
invariants.

\end{abstract}

%
\maketitle

\section{Introduction}
\label{s:intro}
   Due to promising results obtained on invariants of algebraic
functions, and to the similarities in properties between
differential equations and algebraic equations, the study of
invariants of differential equations began in the middle of the
nineteen century. One of the most important studies of these
functions was carried out by Forsyth in his very valuable memoir
~\cite{for-inv}, in which he considers the linear ordinary
differential equation of general order $n$ in various canonical
forms, the first of which is given by
\begin{equation}\label{eq:stdlin}
y^{(n)} + a_{n-1} y^{(n-1)} + a_{n-2} y^{(n-2)}+\dots + a_0 y=0,
\end{equation}
where the coefficients $a_{n-1}, a_{n-2}, \dots, a_0$ are arbitrary
functions of the independent variable $x.$ Earlier writers on the
subject, cited here in an almost chronological order, include
Laplace ~\cite{laplace}, Laguerre ~\cite{lag}, Brioschi
 ~\cite{brioch}, and more importantly Halphen, who made
ground-breaking contributions in his celebrated memoir
 ~\cite{halphen}.\par 
Methods used up to the middle of the nineteen century for studying
invariant functions  were very intuitive, and based on a direct
analysis, in which most results were obtained by a comparison of
coefficients of the equation before and after it was subjected to
allowed transformations. However, the application of these
techniques has been restricted  to linear equations.\par

Based on ideas outlined by Lie ~\cite{lie1}, the development of
infinitesimal methods for the investigation of invariants of
differential equations started probably in ~\cite{ovsy1}, and a
formal  method has been suggested  ~\cite{ibra-not}, which is being
 commonly used  ~\cite{ndog08a, ibra-nl,  ibra-par, waf,
faz}. However, the latter method requires the knowledge of the
equivalence transformations of the equation, and thus a new method
that provides these transformations and at the same time the
infinitesimal generators for the invariant functions has recently
been suggested ~\cite{ndogftc}.\par

In this paper, we use the method of ~\cite{ndogftc} to derive
explicit expressions for the invariants of various canonical forms
of the general linear ordinary differential equations of order up to
5. We also use the same method to derive the structure invariance
group for a certain canonical form of the equation, and for a
general order. Relationships between the invariants found as well as
some of their other properties, including their exact number, are
also investigated. We start our discussion in Section
\ref{s:2methods} by an illustrative example comparing the former
infinitesimal method of
 ~\cite{ibra-not, ibra-nl} with that of ~\cite{ndogftc}.

\section{Two methods of determination}
\label{s:2methods} %
   We begin this section with some generalities about  equivalence transformations.
 Suppose that $\mathcal{F}$ represents a
family of differential equations of the form
\begin{equation}\label{eq:delta}
    \Delta(x, y_{(n)}; C)=0,
\end{equation}
where $x=(x^1, \dots, x^n)$ is the set of independent variables, $y=
(y_1, \dots, y_q)$ is the set of dependent variables and $y_{(n)}$
denotes the set of all derivatives of $y$ up to the order $n,$ and
where $C$ denotes collectively the set of all parameters specifying
the family element in $\mathcal{F}.$  These parameters might be
either arbitrary functions of $x,\, y$ and the derivatives of $y$ up
to the order $n,$ or arbitrary constants. Denote by $G$ a connected
Lie group of point transformations of the form
\begin{equation}\label{eq:eqvgp}
x= \phi (\bar{x}, \bar{y}; \sigma), \qquad y = \psi (\bar{x},
\bar{y}; \sigma),
\end{equation}
where $\sigma$ denotes collectively the parameters specifying the
group element in $G.$  We shall say that $G$ is the {\em equivalence
group} of ~\eqref{eq:delta} if it is the largest group of
transformations that maps every element of $\mathcal{F}$ into
$\mathcal{F}.$  In this case ~\eqref{eq:eqvgp} is called the {\em
structure invariance group} of ~\eqref{eq:delta} and the transformed
equation takes  the same form
\begin{subequations}\label{eq:delta2}
\begin{align}
&\Delta (\bar{x}, \bar{y}_{(n)}; \bar{C})=0,\\[-5mm]
\intertext{where \vspace{-4mm}}
&\bar{C}_j = \bar{C}_j(C, C_{(s)}; \sigma), \label{eq:coef}
\end{align}
\end{subequations}
and where $C_{(s)}$ represents the set of all derivatives of $C$ up
to a certain order $s.$  In fact, the latter equality
~\eqref{eq:coef} defines another group of transformations $G_c$ on
the set of all parameters $C$ of the differential equation
~\cite{liegp}, and we shall be interested in this paper in the
invariants of $G_c.$ These are functions of the coefficients of the
original equation which have exactly the same expression when they
are also formed for the transformed equation. We also note that $G$
and $G_c$ represent the same set, and we shall use the notation
$G_c$ only when there is  need to specify the group action on
coefficients of the equation.\par
   The terminology used for invariants of differential equations and
their variants in the current literature ~\cite{ndog08a, ibra-par,
waf, faz, melesh} is slightly different from that of Forsyth
~\cite{for-inv} and earlier writers on the subject. What is now
commonly called semi-invariants are functions of the form $\Phi(C,
C_{(r)})$ that satisfies a relation of the form $\Phi(C, C_{(r)}) =
\mathbf{w}(\sigma)\cdot \Phi (\bar{C}, \bar{C}_{(r)}).$ When the
weight $\mathbf{w}(\sigma)$ is equal to one, $\Phi(C, C_{(r)}),$ is
called an invariant, or an absolute invariant. Semi-invariants
usually correspond to partial structure invariance groups obtained
by letting either the depend variable or the independent variables
unchanged.\par

Because the infinitesimal method proposed in ~\cite{ndogftc} is
still entirely new and has been, in particular, applied only to a
couple of examples, we wish to illustrate a comparison of this
method with the most well-known method of ~\cite{ibra-not}. For this
purpose we consider Eq. ~\eqref{eq:stdlin} with $n=3,$ which is the
lowest order for which a linear ODE may have nontrivial invariants.
The structure invariance group of ~\eqref{eq:stdlin} is given by the
change of variables $x= f(\bar{x}), \; y= T(\bar{x}) \bar{y},$ where
$f$ and $T$ are arbitrary functions, and for $n=3,$ this equation
takes the form
\begin{equation}\label{eq:d3gn}
y^{(3)} + a_2 y'' + a_1 y' + a_0 y=0,
\end{equation}
and it is easy to see that the corresponding infinitesimal
transformations associated with  the structure invariance group may
be written in the form
\begin{equation}\label{eq:infid3}
\bar{x} \approx x+ \varepsilon \xi(x), \qquad \bar{y} \approx y+
\varepsilon (\eta(x) y),
\end{equation}
where $\xi$ and $\eta$ are some arbitrary functions. If we now let
$$V_a^{(3)}= \xi \pd_x + y\, \eta \pd_y + \zeta_1 \pd_{y'} + \zeta_2 \pd_{y''}
           +\zeta_3 \pd_{y'''} $$
denote the third prolongation of $ V_a= \xi \pd_x + y \eta \pd_y,$
then we may write

\begin{equation}
\label{eq:yppp} \bar{y}\,'\approx y' + \varepsilon\, \zeta_1 (x, y,
y'), \quad   \bar{y}\,'' \approx y''+\varepsilon\, \zeta_2 (x, y,
y', y''), \quad \bar{y}\,^{(3)}\approx  y^{(3)} + \varepsilon
\,\zeta_3(x, y_{(3)}).
\end{equation}

According to the former method ~\cite{ibra-not}, in order to find
the corresponding infinitesimal transformations of the coefficients
$a_2, a_1,$ and $a_0$ of ~\eqref{eq:d3gn}, and hence to obtain the
infinitesimal generator for the associated induced group $G_c,$ the
original dependent variable $y$ and its derivatives are to be
expressed in terms of $\bar{y}$ and its derivatives according to
approximations of the form
\begin{equation}
\label{eq:yppp2}   y' \approx \bar{y}\,' - \varepsilon\, \zeta_1(x,
y, \bar{y}\,'), \quad y''\approx \bar{y}\,'' - \varepsilon\,\zeta_2
(x,y, \bar{y}\,', \bar{y}\,''), \quad y^{(3)}\approx \bar{y}\,^{(3)}
- \varepsilon \,\zeta_3(x,y, \bar{y}_{(3)}).
\end{equation}
These expressions are then used to substitute $\bar{y}$ and its
derivatives for $y$ and its derivatives in the original differential
equation. In certain cases such as that of simple infinitesimal
transformations of the form ~\eqref{eq:infid3}, the required
approximations ~\eqref{eq:yppp2} can be obtained from
 ~\eqref{eq:yppp}. More precisely, an explicit calculation of
$V_a^{(3)}$ shows that
\begin{subequations}\label{eq:zetad3}
\begin{align}
\zeta_1  &= y g' + (g-f')y' \\
\zeta_2  &= y'(2 g'-f'') + y g'' + (g- 2f')y''\\
\zeta_3  &= 3(g'-f'')y''+ y'(3 g'' - f''')+ y g'''+ (g-3f')y'''
\end{align}
\end{subequations}

The first explicit approximation is readily obtained from the second
equation of ~\eqref{eq:infid3} which shows that
\begin{equation}\label{eq:rly}
y= \bar{y} \frac{1}{1+ \varepsilon g}= \bar{y} (1- \varepsilon g),
\end{equation}
by neglecting terms of order two or higher in $\varepsilon.$ In a
similar way, using the equations ~\eqref{eq:yppp2} and
~\eqref{eq:zetad3} we obtain after simplification the following
approximations
\begin{subequations}\label{eq:rlyppp}
\begin{align}
y' &=  (\bar{y}\,' - \varepsilon y g')(1- \varepsilon (g-f'))\\
y'' &=   \bigl[ \bar{y}\,'' - \varepsilon \left( y' (2 g'-f'')+ y g'' \right)\left(1- \varepsilon (g- 2 f') \right) \bigr]\\
y''' &= \left[ \bar{y}\,^{(3)} - \varepsilon \left( 3 (g'-f'')y''+
y' (3 g'' - f''') + y g^{(3)} \right) \right] \left( 1- \varepsilon
(g- 3 f') \right)
\end{align}
\end{subequations}

Substituting equations ~\eqref{eq:rly} and ~\eqref{eq:rlyppp} into
~\eqref{eq:d3gn} and rearranging shows that the corresponding
infinitesimal transformations for the coefficients of the equation
are given by
\begin{subequations}\label{eq:ifcofd3}
\begin{align}
\bar{a}_0 & = a_0 + \varepsilon \bigl[ - 3 a_0 f'- a_1 g' - a_2 g'' - g^{(3)} \bigr]  \\
\bar{a}_1 &= a_1 + \varepsilon \bigl[  - 2 a_1 f' - 2 a_2 g'+ a_2 f'' - 3 g'' + f^{(3)}\bigr]\\
\bar{a}_2 &= a_2 + \varepsilon \bigl[ - a_2 f' - 3g' + 3 f'' \bigr].
\end{align}
\end{subequations}
In other words, the infinitesimal generator of the group $G_c$
corresponding to ~\eqref{eq:d3gn} is given by
\begin{equation}\label{eq:generic3}
 \begin{split}  X^0=&f \pd_x \\
                   &+ (  - a_2 f' - 3g' + 3 f'' )  \pd_{a_2} \\
                   &+  (- 2 a_1 f' - 2 a_2 g'+ a_2 f'' - 3 g'' + f^{(3)}t) \pd_{a_1}\\
                   &+ (- 3 a_0 f'- a_1 g' - a_2 g'' - g^{(3)} ) \pd_{a_0}
\end{split}
\end{equation}

The other method that has just been proposed in ~\cite{ndogftc} for
finding the infinitesimal generators $X^0$ of the group $G_c$ for a
given differential equation simultaneously determines the
equivalence group of the equation, in infinitesimal form. The first
step
 in the determination of $X^0$ is, according to that method, to look
for the infinitesimal generator $X$ of ~\eqref{eq:d3gn}, in which
the arbitrary coefficients are also considered as dependent
variables. For a general equation of the form ~\eqref{eq:delta}, $X$
takes the form
$$
X= \xi_1  \pd_{x^1} +\dots + \xi_p  \pd_{x^p} + \eta_1  \pd_{y_1} +
\dots + \eta_q  \pd_{y_q} + \phi_1  \pd_{C^1} + \phi_2  \pd_{C^2} +
\dots ,
$$
where $(C^1, C^2, \dots)=C$ denotes the set of  parameters
specifying the family element in $\mathcal{F}.$  This expression for
$X$ may be written in the shorter form $X= \set{\xi, \eta, \phi},$
where $\xi, \eta$ and $\phi$ represents collectively the functions
$\xi_j, \eta_j$ and $\phi_j$, respectively. The next step according
to the method is to look at minimum conditions that reduce $V=
\set{\xi, \eta}$ to a generator $V^0 = \set{\xi^0, \eta^0}$  of the
group $G$ of equivalence transformations. These conditions are also
imposed on $\phi$ to obtain the resulting function  $\phi^0,$ and
hence the generator $X^0= \set{\xi^0, \phi^0}$ of $G_c.$\par

In the actual case of equation ~\eqref{eq:d3gn} we readily find that
the generator
\begin{equation}\label{eq:symd3gn}
X= \xi  \pd_x + \eta  \pd_y + \phi_2  \pd_{a_2}  + \phi_1 \,
\pd_{a_1} + \phi_0  \pd_{a_0}
\end{equation}
is given by
\begin{align*}
\xi  &= f \\
\eta &= g y+ h\\
\phi_2 &= - a_2 f' - 3g' + 3 f''\\
\phi_1 &=  - 2 a_1 f'- 2a_2 g' + a_2 f'' - 3 g'' + f^{(3)}\\
\phi_0 &=- \frac{1}{y}\left[   a_0 h + 3 a_0 y f' + a_1 y g' + a_1
h'x + a2 y g'' + a_2 h '' + y g^{(3)} + h^{(3)} \right],
\end{align*}
where $f,g$ and $h$ are arbitrary functions of $x.$ It is clear
that, due to the homogeneity of ~\eqref{eq:d3gn}, we must have
$h=0.$ This simple condition immediately reduces $V= \set{f, y g+h}$
to the well-known generator $V^0= \set{f, y g}$ of $G.$ If we denote
by $\phi^0= \set{\phi_2^0, \phi_1^0, \phi_0^0}$ the resulting vector
when the same condition $h=0$ is applied to $\phi= \set{\phi_2,
\phi_1, \phi_0},$ then $X^0= \set{f, \phi^0},$ which may be
represented as
\begin{equation}\label{eq:ifd3gnb}
X^0= f \pd_x + \phi_2^0 \pd_{a_2} + \phi_1^0 \pd_{a_1} + \phi_0^0
\pd_{a_0},
\end{equation}

is the same as the operator $X^0$ of ~\eqref{eq:generic3}.\par

 It
should however be noted that the former method involves a great deal
of algebraic manipulations, and in the actual case of Eq.
~\eqref{eq:d3gn}, the method works because of the simplicity of the
infinitesimal generator $V_a= \set{f, y g}.$ For higher orders of
the Eq. ~\eqref{eq:stdlin}, difficulties with this method become
quite serious. In fact, variants of this former method are being
increasingly used, which also exploit the symmetry properties of the
equation ~\cite{ibra-par}, but even these variants still require the
knowledge of the equivalence group.

\section{Invariants for equations of lowest orders}
We shall use the indicated method of ~\cite{ndogftc} in this section
to derive explicit expressions for invariants of linear ODEs of
order not higher than the fifth. It should be noted that, not only
the infinitesimal method we are using for finding these invariants
is new, but also almost all of the invariants found are appearing in
explicit form for the first time. The focus of Forsyth and his
contemporaries was not only on linear equations of the form
~\eqref{eq:delta}, but also on functions $\Phi(C, C_{(r)})$
satisfying conditions of the form

\begin{equation}\label{eq:forminv}
\Phi(C, C_{(r)})= h(g)\cdot \Phi(\bar{C}, \bar{C}_{(r)}), \quad
\text{ or } \quad \Phi(C, C_{(r)})=  (d \bar{x}/ d x)^\sigma \cdot
\Phi(\bar{C}, \bar{C}_{(r)}),
\end{equation}
in the case of linear ODEs, where $h$ is an arbitrary function,
$\sigma$ is a scalar and where
 $g= g(\bar{x})$ and $\bar{x}$ are defined by a change of variables
of the form $x= f(\bar{x}),\; y= g\, \bar{y},$ where $f$ is
arbitrary. Such functions are clearly invariants of the linear
equation only if the dependent variable alone, or the independent
variable alone, is transformed, but not both as in our analysis.
However, in his very talented analysis starting with the
investigation of semi-invariants satisfying the second condition of
~\eqref{eq:forminv}, Forsyth obtained in his memoir ~\cite{for-inv}
a very implicit expression for true invariants of linear ODEs of a
general order, given as an indefinite sequence  in terms of
semi-invariants. As it is stated in ~\cite{for-inv}, earlier works
on the subject had given rise only to semi-invariants of order not
higher than the fourth, with the exception of two special functions
which satisfy the condition of invariance for equations of all
orders.\par

\subsection{Equations in normal reduced form}
Invariants of differential equations generally have a prominent size
involving several terms and factors, and thus they are usually
studied by first putting the equation in a form in which a number of
coefficients vanish. By a change of variable of the form
\begin{equation}\label{eq:chg2nor}
x=\bar{x},\qquad y = \exp(- \int a_1 \,d \bar{x}) \bar{y},
\end{equation}
equation ~\eqref{eq:stdlin} can be reduced to the simpler form
\begin{equation}\label{eq:nor}
y^{(n)} + A_{n-2} y^{(n-2)}+  A_{n-3} y^{(n-3)}+ \dots + A_0 y=0,
\end{equation}
after the renaming of variables, and where the $A_j= A_j(x)$ are the
new coefficients. Eq. ~\eqref{eq:nor}, which is referred to as the
normal form of ~\eqref{eq:stdlin},  will be as in ~\cite{for-inv}
our starting canonical form for the determination of invariants.

The generic infinitesimal generator $X^0$ of  $G_c$ such as the one
found in ~\eqref{eq:generic3} linearly depends on free parameters
$K_j,$ and can be written in terms of these parameters as a linear
combination of the form,
\begin{equation}
\label{eq:linearcomb} X^0= \sum_{j=1}^\nu K_j W_j,
\end{equation}
 where the $W_j$ are much simplified vector fields free of
 arbitrary parameters and with the same number of
 components as $X^0,$ and a function $F=F(C)$ satisfies the
 condition of invariance $X^0 \cdot F=0$ if and only if $W_j \cdot
 F=0,$ for $j=1, \dots, \nu.$ Consequently, the invariant functions
 are completely specified by the matrix $\mathcal{M}\equiv \set{W_1, \dots,
 W_\nu}$ whose $j$th row is represented by the components of $W_j,$
 plus the coordinate system in which the vector fields $W_j$ are
 expressed. These considerations also apply to prolongations of
 $X^0$ and corresponding differential invariants. For a given canonical
form of the equation, we will use the symbols $\Psi_{a}^{b,c}$ and
$X_a^m$ to represent explicit expressions of invariants and
corresponding infinitesimal generator $X^0,$ respectively. In such a
representation, the subscript $a$ will denote the order of the
equation, while the superscript $b$ will represent the order of
prolongation of the original generator $X^0,$ and $c$ will represent
the actual number of the invariant in some chosen order. The
superscript $m$ will represent the canonical form considered, and
will be $\mathsf{n}$  in the case of the normal form
~\eqref{eq:nor}, $\mathsf{s}$ for the standard form
~\eqref{eq:stdlin}, and $\mathsf{w}$ for another canonical form to
be introduced below. For consistency, coefficients in all canonical
forms considered will be represented as in ~\eqref{eq:stdlin} by the
symbols $a_j= a_j(x).$\par

For $n=3,$ the first nontrivial differential invariants of
~\eqref{eq:nor} occurs only as from the third prolongation of $X^0.$
This third prolongation of the generator $X^0$ has exactly one
invariant $\Psi_3^{3,1},$ and both $X^0=X_3^\mathsf{n}$ and
$\Psi_3^{3,1}$ have already appeared in the recent literature
~\cite{ndogftc}. However, we recall these results here in a slightly
simplified form, together with those we have now obtained for the
fourth prolongation of $X_3^\mathsf{n},$ as well as for the case
$n=4.$
\par

For $n=3,$ we have
\begin{subequations} \label{eq:casnord3}
\begin{align}
X_3^\mathsf{n} &= f \pd_x -2 (a_1 f' + f^{(3)}) \pd_{a_1} + \left(-
3 a_0 f' - a_1 f''
- f^{(4)}\right)\pd_{a_0}  \\
\Psi_3^{3,1} &=  \frac{ (9\, a_1
\mu^2+ 7 \mu'\,^2- 6 \mu \mu'')^3}{1728 \mu^8}\\
\Psi_3^{4,1} &= \Psi_3^{3,1} \\
\begin{split} \Psi_3^{4,2} &=  \frac{-1}{18 \mu^4} \left( 216 a_0^4 - 324 a_0^3\, a_1' + 18 \, a_0^2 (9 a_1'^2+ 2 a_1
\mu') + + 9 \mu^2 \mu^{(3)}\right)\\
 & \frac{-1}{18 \mu^4} \left(\mu' (28 \mu'\,^2 + 9 a_1'(a_1\, a_1' - 4 \mu'')) - 9 a_0(3 a_1'^3+ 4 a_1 \, a_1' \mu'- 8 \mu' \mu''), \right)\end{split}
\end{align}
\end{subequations}
where $f$ is an arbitrary function and where $\mu= -2a_0+ a_1'.$ For
$n=4,$ the first nontrivial invariant occurs only as from the second
prolongation of $X^0= X_4^\mathsf{n},$ and we have
\begin{subequations}\label{eq:norfn4}
\begin{align}
 \begin{split} X_4^\mathsf{n} &= f\pd_x +(- 2 a_2 f' - 5 f^{(3)}) \pd_{a_2} + (-3 a_1
f' - 2 a_2 f'' - 5 f^{(4)})\pd_{a_1}\\
 & \quad +\frac{1}{2} \left[  -8\, a_0
f' - 3 (a_1 f'' + a_2 f^{(3)} )\right]
\pd_{a_0}\end{split}  \\
\Psi_4^{2,1} &= \frac{1}{\mu} (-100 a_0 + 9 a_2^2 + 50 a_1' - 20 a_2'') \\
\begin{split} \Psi_4^{2,2} &= \mu_1 (-1053 a_2^4 - 4500
a_2(2 a_1^2 - a_1 a_2' - a_2'^{\,2}) + 180 a_2^2 (130 a_0 - 3 (5
a_1'+ 8\, a_2'')) \\
& - 100 (1300 a_0^2+75 \mu (-10 a_0'+ 3 a_1'') + 90 a_1' a_2'' + 27
a_2''^{\,2}-60 a_0 (5 a_1'+ 8\, a_2'') ), \end{split}
\end{align}
\end{subequations}

 where $\mu= a_1- a_2',$ and $\mu_1= 1/\left[75000
\mu^{(8/3)}\right].$

\subsection{Equations with vanishing coefficients $a_{n-1}$ and
$a_{n-2}$} A much simpler expression for the invariants is obtained
if the two terms involving the coefficients $a_{n-1}$ and $a_{n-2}$
in Eq. ~\eqref{eq:stdlin} are reduced to zero in the transformed
equation. Such a change of variables can be accomplished by a
transformation of the form
\begin{subequations}\label{eq:chg2sch}
\begin{align}
\set{\bar{x}, x}&= \frac{12}{n (n-1)(n+1)}  a_{n-2}, \qquad y =
\exp(- \int a_{n-1} d \bar{x}) \bar{y},\\
\intertext{ where } \set{\bar{x}, x} &= \left(\bar{x}\,'
\bar{x}\,^{(3)} - (3/2) \bar{x}\,''^{\,2} \right)\,
\bar{x}\,'^{\,-2}
\end{align}
\end{subequations}
is the Schwarzian derivative, and where $\bar{x}\,' = d \bar{x}/
dx.$  Thus by an application of ~\eqref{eq:chg2sch} to
~\eqref{eq:stdlin}, we obtain after the renaming of variables and
coefficients an equation of the form
\begin{equation}\label{eq:sch}
y^{(n)} + a_{n-3} y^{(n-3)} + a_{n-4} y^{(n-4)} + \dots + a_0 y=0.
\end{equation}
In fact, ~\eqref{eq:sch} is the canonical form  that Forsyth
 ~\cite{for-inv}, Brioschi ~\cite{brioch}, and some of their
contemporaries adopted for the investigation of invariant functions
of linear ODEs. However, Forsyth who studied these equations for a
general order did not derive any explicit expressions for the
invariants, with the exception of a couple of semi-invariants. There
is a number of important facts that occur in the determination of
the invariants when the equation is put into the reduced form
~\eqref{eq:sch}, but we postpone the discussion on the properties of
invariants to the next Section, where a formal result regarding
their exact number is also given.\par

For $n=3,$ nontrivial invariants exist only as from the second
prolongation of $X^0= X_3^\mathsf{w},$ and we have computed them for
the second and the third prolongation. For $n=4,$ there are no
zeroth order invariants and we have computed these invariants for
the first and the second prolongation of $X_4^\mathsf{w}.$ For
$n=5,$ the invariants are given for all orders of prolongation of
$X_5^\mathsf{w},$ from $0$ to $2.$ Since invariants corresponding to
a given generator are also invariants for any prolongation of this
generator, its is only necessary to list the invariants for the
highest order of prolongation of any given generator. In the
canonical form ~\eqref{eq:sch}, and regardless of the order of the
equation, the generator $X^0$ of $G_c$ will depend on three
arbitrary constants that we shall denote by $k_1, k_2$ and
$k_3.$\par
For $n=3,$ we have
\begin{align*}
X_3^\mathsf{w} &= \left[ k_1+ x (k_2 + k_3 x)\right] \pd_x - 3 a_0
(k_2+ 2
k_3 x)\pd_{a_0}\\
\Psi_3^{3,1} &= \Psi_3^{2,1}= (-\frac{7 a_0'\,^2}{6 a_0} + a_0'')^3
/a_0^5\\
\Psi_3^{3,2} &= \frac{(28 a_0'^3 - 36 a_0\, a_0'\,  a_0'' + 9
a_0^2\, a_0^{(3)})}{9 a_0^4}
\end{align*}
For $n=4,$ we have
\begin{align*}
X_4^\mathsf{w} &= \left[ k_1+ x (k_2 + k_3 x)\right] \pd_x - 3 a_1 (k_2 + 2 k_3 x)\pd_{a_1} + \left[-3 a_1 k_3 - 4 a_0 (k_2 + 2 k_3 x)\right]\pd_{a_0}   \\
\Psi_4^{2,1} &= \Psi_4^{1,1}= \frac{(- 2 a_0 + a_1')^3}{a_1^4}  \\
\Psi_4^{2,2} &= \Psi_4^{1,2} = \frac{(a_0' - \frac{a_0 (a_0 + 3 a_1')}{3 a_1})^3}{a_1^5} \\
\Psi_4^{2,3} &= \frac{(14 a_0^2 - 14 a_0\, a_1' + 3 a_1\, a_1'')^3}{27 a_1^8}\\
\Psi_4^{2,4} &= \frac{16 a_0^3 + 48 a_0^2\, a_1'+ 9 a_1^2\, a_0'' -
9 a_0\, a_1(6 a_0' + a_1'')}{9 a_1^4}
\end{align*}
Finally, for $n=5$ we have
\begin{subequations}
\begin{align}
\begin{split}X_5^\mathsf{w} &= \left[ k_1+ x (k_2 + k_3 x)\right] \pd_x - 3 a_2 (k_2 + 2 k_3 x)\pd_{a_2}\notag \\
 & \quad -2 \left[-3 a_2 k_3 +2 a_1 (k_2 + 2 k_3 x)\right]\pd_{a_1} + \left[ -4 a_1 k_3 - 5 a_0 (k_2+ 2 k_3 x)  \right] \pd_{a_0}\end{split}   \notag \\
\Psi_5^{2,1} &= \Psi_5^{0,1}= \Psi_5^{1,1}=  \frac{\left(3a_0\, a_2 - a_1^2  \right)^3}{27 a_2^8} \label{eq:schp0}\\
\Psi_5^{2,2} &= \Psi_5^{1,2} = \frac{(- a_1 + a_2')^3}{a_2^4}  \notag \\
\Psi_5^{2,3} &= \Psi_5^{1,3} =  \frac{\left( 6 a_2 a_1' - a_1^2 - 6 a_1\, a_2'\right)^3}{216\, a_2^8}\notag \\
\Psi_5^{2,4} &= \Psi_5^{1,4} =\frac{5 a_1^3 + 9 a_2^2\, a_0' - 3
a_1\, a_2(5 a_0  + 2 a_1')+ 3 a_1^2 \, a_2'}{9 a_2^4}
\end{align}
\end{subequations}
and
\begin{align*}
\Psi_5^{2,5} &= \frac{(7 a_1^2 - 14 a_1\, a_2' + 6 a_2\, a_2'')^3}{216\, a_2^8}  \\
\Psi_5^{2,6} &= \frac{4 a_1^3 + 24 a_1^2 \, a_2'+ 9 a_2^2 a_1'' - 9 a_1 a_2(3 a_1'+ a_2'')}{9 a_2^4}\\
\Psi_5^{2,7} &= \frac{\left[18(- a_1^4 - a_1^3\, a_2'+ a_2^3\,
a_0'')- 6 a_1 a_2^2 (11 a_0' + 2 a_1'') + a_1^2 a_2(55 a_0+ 40 a_1'
+ 6 a_2'') \right]^3}{5832\, a_2^{16}}
\end{align*}

\subsection{Equations in the standard form $~\eqref{eq:stdlin}$}
No attempt has ever been made to our knowledge, to obtain the
invariant functions for equations in the canonical form
~\eqref{eq:stdlin}, due simply to difficulties associated with such
a determination and to the prominence in size that semi-invariants
for this canonical form already display. Indeed, semi-invariants for
the canonical form ~\eqref{eq:stdlin} were calculated by Laguerre
 ~\cite{lag} for equations of the third order. Transformations of the
form ~\eqref{eq:chg2nor} or ~\eqref{eq:chg2sch} are usually applied
to transform ~\eqref{eq:stdlin} to an equation in which one or both
of the coefficients $a_{n-1}$ and $a_{n-2}$ vanish. Even with the
infinitesimal generator $X^0$ of ~\eqref{eq:generic3} at our
disposal, it is still difficult to find directly the corresponding
invariants for the third order equation, which is the lowest for
which the linear equation may have nontrivial invariants. However we
show here by an example how these invariants can be found for the
third order from those of equations in a much simpler canonical
form.\par

Under the change of variables ~\eqref{eq:chg2nor}, Eq.
~\eqref{eq:d3gn} takes, after elimination of a constant factor, the
form
\begin{align}\label{eq:red2nord3}
&\bar{y}^{(3)} + B_1 \bar{y}\,' + B_0 \bar{y}=0,  \\[-4mm]
\intertext{where \vspace{-5mm}}
&B_0 = (27 a_0 - 9 a_1\, a_2 + 2 a_2^3 - 9 a_2'')/27\\
&B_1 = (27 a_1 - 9 a_2^2 - 27 a_2')/27.
\end{align}
The sole invariant $\Psi_3^{3,1}$ for equations of the form
~\eqref{eq:red2nord3} and corresponding to the third prolongation of
the associated generator $X^0$ is given by ~\eqref{eq:casnord3}. If
in that function we replace $a_0$ and $a_1$ by $B_0$ and $B_1$
respectively, and then express $B_0$ and $B_1$ in terms of the
original coefficients $a_0, a_1$ and $a_2,$ the resulting function
is an invariant of ~\eqref{eq:d3gn}. However, it does not correspond
to the third, but to the fourth prolongation of the corresponding
generator $X^0= X_3^\mathsf{s}$  that was obtained in
~\eqref{eq:generic3}. In other words, in the canonical form
~\eqref{eq:stdlin} with $n=3,$ we have
\begin{align}
X_3^\mathsf{s} &= X^0, \;\text{as given by ~\eqref{eq:generic3}}  \notag \\
\Psi_3^{4,1} &=   \frac{ (9\, B_1 \mu^2+ 7 \mu'^2- 6 \mu
\mu'')^3}{1728\, \mu^8 }, \label{eq:ivd3gn}
\end{align}
where we have $\mu = -2 B_0 + B_1'.$  Using Theorem 2 of
 ~\cite{ndogftc}, it can be shown that there are no invariants
corresponding to prolongations of $X_3^\mathsf{s}$ of order lower
than the fourth, and that the fourth prolongation has precisely one
invariant.
\section{A note about the infinitesimal generators and the invariants}
\label{s:note} The method of ~\cite{ndogftc} that we have used in
the previous section provides the infinitesimal generator of the
equivalence group $G$ for a given equation of a specific order. This
infinitesimal generator must be integrated in order to obtain the
corresponding structure invariance group for the given family of
equations of a specific order. However, once the structure
invariance group for such a given equation has been found, it is
generally not hard to extend the result to a general order. We show
how this can be done for equations of the form ~\eqref{eq:sch}.\par

   With the notations already introduced in Section \ref{s:2methods},
if we denote by $X$ the full symmetry generator of ~\eqref{eq:sch}
for a specific order and then obtain the infinitesimal generator
$V^0$ of the equivalence group $G,$ we find that
\begin{subequations}
\begin{alignat}{2}\label{eq:ifshd3-5}
V^0= V_{0,3} &=  \left[ k_2 + x (k_3+ k_4 x) \right] \pd_x +   y (k_1+ 2 k_4 x) \pd_y,&& \quad \text{ for $n=3,$}\\
V^0= V_{0,4} &=  \left[ k_2 + x (k_3+ k_4 x) \right] \pd_x +    \frac{y}{3} (2 k_1+ 2 k_3 + 9 k_4 x)\pd_y,&& \quad \text{ for $n=4,$}\\
V^0= V_{0,5} &=  \left[ k_2 + x (k_3+ k_4 x) \right] \pd_x + y (k_1
+ k_3 + 4 k_4 x)\pd_y,&& \quad \text{ for $n=5,$}
\end{alignat}
\end{subequations}
where the $k_j, \text{for $j=1, \dots,4$}$  are arbitrary constants.
Upon integration of these three vector fields, we find that for
$n=3,4,5,$ the structure invariance group can be written in the
general form
\begin{equation}\label{eq:sigp3-5}
x= \frac{\bar{x} - c(1+ a \bar{x})}{b(1+ a\bar{x})}, \qquad y =
\frac{\bar{y}}{d(1+ a \bar{x})^{(n-1)}},
\end{equation}
where $a,b, c$ and $d$ are arbitrary constants. To see why
~\eqref{eq:sigp3-5} holds for any order $n$ of Eq. ~\eqref{eq:sch},
we first notice that these changes of variables are of the much
condensed  form
$$
x=g (\bar{x}), \qquad y = T(\bar{x}) \bar{y}.
$$
It thus follows from the properties of derivations that when
~\eqref{eq:sigp3-5} is applied to ~\eqref{eq:sch}, every term of
order $m$ in $y$ will involve a term of order at most $m$ in
$\bar{y}$ upon transformation. Since the original equation
~\eqref{eq:sch} is deprived of terms of orders $n-1$ and $n-2,$ to
ensure that this same property will also holds in the transformed
equation, we only need to verify that the transformation of the term
of highest order, viz.  $y^{(n)},$ does not involve terms in
$\bar{y}^{(n-1)}$ or $\bar{y}^{(n-2)}.$ However, it is easy to see
that under ~\eqref{eq:sigp3-5}, $y^{(n)}$ is transformed into $b^n
(1+ a \bar{x})^{n+1} \bar{y}^{(n)}/d.$  We have thus obtained the
following result.
\begin{thm}\label{th:sigpsch}
The structure invariance group of linear ODEs in the canonical form
~\eqref{eq:sch}, that is, in the form
$$ y^{(n)} + a_{n-3} y^{(n-3)} + a_{n-4} y^{(n-4)} + \dots + a_0 y=0,$$
is given for all values of $n>2$ by ~\eqref{eq:sigp3-5}.
\end{thm}
It is easy to prove a similar result using the same method for
equations in the other canonical forms ~\eqref{eq:stdlin} and
~\eqref{eq:nor}, for which the corresponding structure invariance
groups are well-known. It is also clear that ~\eqref{eq:sigp3-5}
does not violate in particular the structure invariance group of
~\eqref{eq:nor} which is given ~\cite{schw} by $x= F (\bar{x})$ and
$y= \alpha F'^{(n-1)/2} \bar{y},$ where $F$ is an arbitrary function
and $\alpha$ an arbitrary coefficient.
\par
All the invariants that we have found for linear ODEs  are rational
functions of the coefficients of the corresponding equation and
their derivatives, and this agrees with earlier results
 ~\cite{for-inv}. Another fact about these invariants is that for
equations in the form ~\eqref{eq:sch}, nontrivial invariants that
depend exclusively on the coefficients of the equation and not on
their derivatives, that is, invariants such as that in
~\eqref{eq:schp0} which are zeroth order differential invariants of
$X^0$ occurs as from the order $5$ onwards. Indeed, since $X^0$
depends only on three arbitrary constants in the case of the
canonical form ~\eqref{eq:sch}, there will be more invariants, as
compared to the other canonical forms in which $X^0$ depends on
arbitrary functions. More precisely, we have the following result
\begin{thm}
For equations  in the canonical form ~\eqref{eq:sch}, the exact
number $\mathbf{\small n}$ of fundamental invariants for a given
prolongation of order $p \geq 0$ of the infinitesimal generator
$X^0$ is given by
\begin{equation}\label{eq:nbinvsch}
\mathbf{\small n}=\begin{cases} n + p\,(n-2) -4, & \text{if $(n,p)
\neq (3,0)$} \\
0, & \text{otherwise}.
\end{cases}
\end{equation}
\end{thm}
\begin{proof}
We first note that the lowest order for equations of the form
 ~\eqref{eq:sch} is $3,$ and we have seen that when $n=3,4,5$ the
generator $X^0$ depends on three arbitrary constants alone, that we
have denoted by $k_1, k_2$ and $k_3,$ and do not depend on any
arbitrary functions. Consequently, this is also true for any
prolongation of order $p$ of $X^0,$  and it is also not hard to see
that this property does not depend on the order of the equation. It
is clear that the corresponding matrix $\mathcal{M}$ whose $j$th row
is represented by the components of the vector field $W_j$ of
~\eqref{eq:linearcomb} will always have three rows and maximum rank
$\mathsf{r} =3$ when $(n,p) \neq (3,0).$ On the other hand, $X^0$ is
expressed in a coordinate system of the form $\set{x, a_{n-3},
a_{n-4}, \dots, a_0}$ that has $n-1$ variables. Its $p$th
prolongation is therefore expressed in terms of $M=n-1 + p(n-2)$
variables. For $(n,p)\neq (3,0),$ the exact number of fundamental
invariants is $M-\mathsf{r},$ which is $n + p(n-2) -4,$ by Theorem 2
of
 ~\cite{ndogftc}. For $(n,p)=(3,0)$ a direct calculation shows that
the number of invariants is zero. This completes the proof of the
Theorem.
\end{proof}
Similar results can be obtained for equations in other canonical
forms, but as it is customary to choose the simplest canonical form
for the study of invariants, we shall not attempt to prove them
here. \par
It should be noted at this point that invariant functions have
always been intimately associated with important properties of
differential equations, and in particular with their integrability.
Laguerre \cite{lag} showed that for the third order equation
\eqref{eq:d3gn}, when the semi-invariant $\mu= - 2 B_0+ B_1'$ given
in Eq. \eqref{eq:ivd3gn} vanishes, there is a quadratic homogeneous
relationship between any three integrals of this equation, and
consequently its integrability is reduced to that of a second order
equation. Subsequently, Brioschi \cite{brioch} established a similar
result for equations in the canonical form \eqref{eq:nor}. More
precisely, he showed that for third order equations, when the
semi-invariant $\mu= -2 a_0+ a_1'$ of Eq. \eqref{eq:casnord3}
vanishes, the equation can be reduced in order by one to the second
order equation
$$ \bar{y}'' + \frac{a_1}{4} \bar{y}=0 ,$$
where $y= \bar{y}^2$ and he also gave the  counterpart of this
result for fourth order equations when the corresponding
semi-invariant $\mu=a_1-a_2'$ of Eq. \eqref{eq:norfn4} vanishes. In
fact all reductions of differential equations to integrable forms
implemented in \cite{halphen} are essentially based on  invariants
of differential equations. More recently, Schwarz \cite{schw}
attempted to obtain a classification of third order linear ODEs
based solely on values of invariants of these equations. He also
proposed \cite{schw2} a solution algorithm for large families of
Abel's equation, based on what is called in that paper a functional
decomposition of the absolute invariant of the equation.

\section{Concluding Remarks}\label{s:concl}
We have clarified in this paper the algorithm of the novel method
that has just been proposed in ~\cite{ndogftc} for the equivalence
group of a differential equation and its generators, and we have
shown how its application to the determination of invariants differs
from the former and well-known method of
 ~\cite{ibra-not, ibra-nl},
by treating the case of the third order linear ODEs with the two
methods. We have subsequently obtained some explicit expressions of
the invariants for linear ODEs of orders three to five, and
discussed various properties of these invariants and of the
infinitesimal generator $X^0.$\par
 Another fact that has emerged concerning the infinitesimal generator $X^0$
of the group $G_c$ for linear ODEs  is that, whenever the equation
is transformed into a form in which the number of coefficients in
the equation is reduced by one, the number of arbitrary functions in
the generator $X^0$ corresponding to the transformed equation is
also reduced by one, regardless of the order of the equation. Thus
while $X_a^\mathsf{s}$ depends on two arbitrary functions,
$X_a^\mathsf{n}$ depends only on one such function and $X_a^w$
depends on no such function, but on three arbitrary constants. This
is clearly in agreement with the expectation that equations with
less coefficients are easier to solve, since less arbitrary
functions in $X^0$ means more invariants and hence more possibility
of reducing the equation to a simpler one. However, the full meaning
of the degeneration of these functions with the reduction of the
number of  coefficients of the equation is still to be clarified. We
also believe that on the base of recent progress on generating
systems of invariants and invariant differential operators (see
 ~\cite{olvmf, olvgen} and the references therein), it should be
possible to treat the problem of determination on invariants of
differential equations in a more unified way, regardless of the
order of the equation, or the order of prolongation of the operator
$X^0.$

%
\end{document}